\documentclass[reqno,10pt]{amsart}

\title{Hessian of Hausdorff dimension on purely imaginary directions}

\author[]{Martin Bridgeman}
\author[]{Beatrice Pozzetti}

\author[]{Andr\'es Sambarino}

\author[]{Anna Wienhard}
\thanks{A. S. was partially financed by ANR DynGeo ANR-16-CE40-0025. B. P. and  A. W. acknowledge funding by the Deutsche Forschungsgemeinschaft within the Priority Program SPP 2026 ``Geometry at Infinity''. A. W. acknowledges funding  by the European Research Council under ERC-Consolidator grant 614733, and by the Klaus-Tschira-Foundation. M. B. acknowledges funding  by NSF Grants DMS-2005498.}
\date{}

\subjclass[]{}

\usepackage[colorlinks = true,backref = page,hyperindex,breaklinks]{hyperref} % Links in the pdf. Backref option: each bibliographical entry denotes where it was cited
\usepackage{ stmaryrd }
\usepackage{amssymb}
\usepackage{mathtools}      % improves amsmath
\usepackage{mathabx}        % makes many symbols (as $\leq$) more beautiful; must be loaded after mathtools 
\usepackage[bb = fourier,cal = euler,scr = rsfs]{mathalfa}	% selecting fancy math fonts
\usepackage{enumitem}       % special itemize with custom tags and labels that work 
\usepackage[english]{babel}
\usepackage{tikz}			% figures
\usepackage{tikz-cd}		% commutative diagrams
\usepackage[font = small]{caption}	% smaller captions for figures
\usepackage{nicefrac}
\usepackage{comment}

\usetikzlibrary{calc}
\renewcommand*{\backref}[1]{}
\renewcommand*{\backrefalt}[4]{\quad \tiny
  \ifcase #1 (\textbf{NOT CITED.})%
  \or    (Cited on page~#2.)%
  \else   (Cited on pages~#2.)%
  \fi}

\makeatletter

\def\MRbibitem{\@ifnextchar[\my@lbibitem\my@bibitem}

\def\mybiblabel#1#2{\@biblabel{{\hyperref{http://www.ams.org/mathscinet-getitem?mr=#1}{}{}{#2}}}}

\def\myhyperanchor#1{\Hy@raisedlink{\hyper@anchorstart{cite.#1}\hyper@anchorend}}

\def\my@lbibitem[#1]#2#3#4\par{%
  \item[\mybiblabel{#2}{#1}\myhyperanchor{#3}\hfill]#4%
  \@ifundefined{ifbackrefparscan}{}{\BR@backref{#3}}%
  \if@filesw{\let\protect\noexpand\immediate% write to aux-file
    \write\@auxout{\string\bibcite{#3}{#1}}}\fi\ignorespaces%
}

\def\my@bibitem#1#2#3\par{%
  \refstepcounter\@listctr% standard tex item counter for the generic item number
  \item[\mybiblabel{#1}{\the\value\@listctr}\myhyperanchor{#2}\hfill]#3%
  \@ifundefined{ifbackrefparscan}{}{\BR@backref{#2}}%
  \if@filesw\immediate\write\@auxout% write to aux-file
    {\string\bibcite{#2}{\the\value\@listctr}}\fi\ignorespaces%
}

\makeatother
\newcommand{\xqedhere}[2]{%
  \rlap{\hbox to#1{\hfil\llap{\ensuremath{#2}}}}}

\newcommand{\R}{\mathbb{R}} 
\newcommand{\C}{\mathbb{C}}
\newcommand{\N}{\mathbb{N}}
\renewcommand{\P}{\mathbb{P}}

\newcommand{\K}{\mathbb K}
\newcommand{\HH}{\mathbb H}

\newcommand{\eps}{\varepsilon}
\newcommand{\G}{\sf{\Gamma}}    % the group
\newcommand{\Gr}{\cal G}    % Grassmannian

\newcommand{\g}{\gamma}

\newcommand{\bord}{\partial}

\newcommand{\UG}{\sf{U}\G}

\newcommand{\wk}{\check}

\newcommand{\sroot}{{\sf{a}}}
\newcommand{\II}{\mathbf{I}}
\newcommand{\JJ}{\mathbf{J}}
\newcommand{\PP}{\mathbf{P}}

\renewcommand{\sf}[1]{{\mathsf{#1}}}
\renewcommand{\rm}{\mathbf}
\newcommand{\cal}{\mathcal}
\renewcommand{\frak}{\mathfrak}

\newcommand{\Weyl}W

\renewcommand{\sl}{\frak{sl}}
\newcommand{\su}{\frak{su}}
\newcommand{\so}{\frak{so}}
\newcommand{\s}{\frak s}

\newcommand{\Bending}{\ker(d_\rho\conj+\id)}

\DeclareMathOperator{\holder}{Hol}
\newcommand{\Anosov}{\frak X}

\newcommand{\QF}{\mathcal{QF}}%quasi-Fuchsian space

\DeclareMathOperator{\conj}{\tau}

\DeclareMathOperator{\class}{C}

\DeclareMathOperator{\PSL}{{\mathsf{PSL}}}

\DeclareMathOperator{\SO}{{\mathsf{SO}}}
\DeclareMathOperator{\PGL}{{\mathsf{PGL}}}
\DeclareMathOperator{\PO}{{\mathsf{PO}}}
\DeclareMathOperator{\PSO}{{\mathsf{PSO}}}
\DeclareMathOperator{\PSU}{{\mathsf{PSU}}}

\DeclareMathOperator{\id}{id}
\DeclareMathOperator{\Isom}{Isom}

\DeclareMathOperator{\Hff}{{Hf{}f}}
\DeclareMathOperator{\Hom}{{Hom}}

\DeclareMathOperator{\Hess}{Hess}

\DeclareMathOperator{\hitchin}{\mathscr{H}}

\newcommand{\Wedge}{\mathsf{\Lambda}}  % Big wedge (for exterior powers)

\DeclareMathOperator{\Var}{Var}
\DeclareMathOperator{\PSp}{{\mathsf{PSp}}}
\DeclareMathOperator{\PU}{{\mathsf{PU}}}

\newcommand{\bsm}{\left(\begin{smallmatrix}}
\newcommand{\esm}{\end{smallmatrix}\right)}
\newcommand{\bpm}{\begin{pmatrix}}
\newcommand{\epm}{\end{pmatrix}}

\renewcommand{\epsilon}{\varepsilon}

\hypersetup{bookmarksdepth  =  3} % Depth of sections/subs... to have bookmark links in the pdf;
\setlist[enumerate,1]{label = {\upshape(\roman*)},ref = \roman*}
\setlist[enumerate,2]{label = {\upshape(\alph*)},ref = \alph*}

\newtheorem{thmA}{Theorem}

\newtheorem{corA}{Corollary}

\newtheorem*{thm*}{Theorem}
\newtheorem*{cor*}{Corollary}
\newtheorem*{prop*}{Proposition}

\newtheorem{thm}{Theorem}[section]

\newtheorem{cor}[thm]{Corollary}
\newtheorem{lemma}[thm]{Lemma}
\newtheorem{prop}[thm]{Proposition}

\theoremstyle{definition}

\newtheorem*{defi*}{Definition}
\newtheorem{defi}[thm]{Definition}

\theoremstyle{remark}
\newtheorem{obs}[thm]{Remark}

\newcommand{\nocontentsline}[3]{}
\newcommand{\tocless}[2]{\bgroup\let\addcontentsline=\nocontentsline#1{#2}\egroup}

\begin{document}
\begin{abstract}
We extend classical results of Bridgeman-Taylor and McMullen on the Hessian of the Hausdorff dimension on quasi-Fuchsian space to the class of $(1,1,2)$-hyperconvex representations, a class introduced in \cite{PSW1} which includes small complex deformations of Hitchin representations and of $\Theta$-positive representations. We also prove that the Hessian of the Hausdorff dimension of the limit set at the inclusion $\G\to\PO(n,1)\to\PU(n,1)$ is positive definite when $\G$ is co-compact in $\PO(n,1)$ (unless $n=2$ and the deformation is tangent to $\frak X\big(\G, \PO(2,1)\big)$).
\end{abstract}
\maketitle

\tableofcontents

\section{Introduction}

One of the most interesting and well studied metrics on the Teichm\"uller space, the parameter space of hyperbolic structures on a closed surface $S$ of genus $g\geq 2$, is the Weil-Petersson metric, a non-complete Riemannian metric. A celebrated result by B.-Taylor \cite{WP-QF} and McMullen  \cite{mcMWP} gives a geometric interpretation of this metric in terms of dynamical invariants of quasi-Fuchsian representations. 

Recall that the holonomy representation realizes the Teichm\"uller space $\cal T(S)$ as a connected component of the character variety 
$$\frak X\big(\pi_1S,\PSL_2(\R)\big):=\Hom\big(\pi_1S,\PSL_2(\R)\big)/\!/\PSL_2(\R),$$
which, in turn, sits as a totally real submanifold of the complex character variety $\frak X\big(\pi_1S,\PSL_2(\C)\big)$, endowed with the complex structure $J$ induced by the complex structure of the Lie group $\PSL_2(\C)$. A neighborhood of $\cal T(S)$ in the complex character variety is given by quasi-Fuchsian space $\QF(S)$, the set of conjugacy classes of representations 
$\rho:\pi_1S\to\PSL_2(\C)=\Isom_0(\HH^3)$ preserving a convex subset of $\HH^3$ on which they act cocompactly. Any such $\rho$ is thus a quasi-isometric embedding and admits an injective equivariant boundary map $\xi_\rho:\bord\pi_1S\to\C\P^1$ whose image is a Jordan curve.
Given $\rho\in\QF(S)$, we denote by $\Hff(\rho)$ the Hausdorff dimension of this Jordan curve. It is bounded below by $1$ and Bowen showed that $\Hff(\rho)$ equals $1$ precisely when $\rho$ belongs to the Teichm\"uller space \cite{bowen-quasicircles}. The result of  B.-Taylor and McMullen realizes the Weil-Petersson metric by looking at the infinitesimal change of the Hausdorff dimension in purely imaginary directions at a representation $\rho\in \cal T(S) \subset \QF(S)$.
\begin{thm}[B.-Taylor \cite{WP-QF} -McMullen \cite{mcMWP}]\label{t.bt}
For each $\rho\in\cal T(S)$ and every differentiable curve $(\rho_t)_{t\in(-\eps,\eps)}\subset\cal T(S)$ with $\rho_0=\rho$ it holds $$\Hess\Hff(J\dot\rho)=\|\dot\rho\|_{WP}.$$
\end{thm}

In recent years, convex-cocompactness has been generalized from rank 1 to real-algebraic semisimple Lie groups\footnote{(of non-compact type)} $\sf G$ of arbitrary rank, via the concept of \emph{Ano\-sov representations} $\rho:\G\to\sf G_\K,$ where, for $\K=\R$ or $\C,$ $\sf G_\K$ denotes the group of the $\K$-points of $\sf G.$ Specifying a set $\Theta$ of simple roots, let $\sf G_\K/\sf P_\Theta$ be the space of parabolic subgroups of type $\Theta.$ Then $\Theta$-Anosov representations are characterized by admitting a continuous, equivariant, transverse boundary map $\xi^\Theta_\rho:\bord\G\to \sf G_\K/\sf P_\Theta$ with good dynamical properties \cite{Labourie-Anosov, GW-Domains, KLP-Morse, KLP1, BPS, KP}. 
They form open subsets $$\frak X_{\Theta}(\G,\sf G_\K)=\big\{\rho\in\frak X(\G,\sf G_\K):\rho\textrm{ is $\Theta$-Anosov}\big\}$$ of the character variety.

For each $\sroot\in\Theta$ B.-Canary-Labourie-S. \cite{pressure} constructed, using the thermodynamic formalism, an analogue of the Weil-Petersson metric on $\frak X_{\Theta}(\G,\sf G_\K)$, the {\em spectral radius pressure form}  $\PP^{\omega_\sroot}$, where $\omega_\sroot$ is the fundamental weight associated to $\sroot.$ We will recall this construction on Section \ref{b2}.

Probably the best studied space of Anosov representations is the $\PSL_d(\R)$-Hitchin component. Hitchin introduced a special connected component $$\hitchin(S,\sf G_\R)\subset\frak X(\pi_1S,\sf G_\R),$$ when $\sf G_\R$ is moreover center-free and simple split, which in the case of $\PSL_d(\R)$ can be described as the connected component $\hitchin_d(S)\subset\frak X\big(\pi_1S,\PSL_d(\R)\big)$ containing a \emph{Fuchsian representation}, i.e. the composition of the holonomy of a hyperbolic structure with the irreducible representation $\pi_1S\to\PSL_2(\R)\to \PSL_d(\R).$
On the $\PSL_d(\R)$-Hitchin component B.-Canary-Labourie-S. \cite{LiouvillePressure} defined a different pressure form, denoted by $\PP^{\sroot_1},$ to which we will refer here as the \emph{spectral gap pressure form}. They  prove that $\PP^{\sroot_1}$ is non-degenerate on $\hitchin_d(S)$ and extends the Weil-Petersson inner product on Teichm\"uller space, embedded into $\hitchin_d(S)$ as the space of Fuchsian representations.

A corollary of the main result of the paper is a generalization of Theorem \ref{t.bt}. To state the result, we denote by $\Pi$ the set of simple (restricted) roots of $\sf G_\R$ and consider the Hitchin component $\hitchin(S,\sf G_\R)$ as a subset of $\frak X_{\Pi}\big(\pi_1S,\sf G_\C\big),$ the latter equipped the complex structure $J$ induced by the complex structure of $\sf G_\C.$ For $\sroot\in\Pi$ denote by $$\Hff_\sroot(\rho)=\dim_{\Hff}\big(\xi^{\sroot}_\rho(\bord\pi_1S)\big)$$ the Hausdorff dimension of the (image of the) limit curve $\xi^{\sroot}_{\rho}:\bord\G\to\cal F_{\sroot}(\sf G_\C)$ for a(ny) Riemannian metric on $\cal F_{\sroot}(\sf G_\C).$ It follows from Theorem \ref{h->h} that $\Hff_\sroot$ is critical at $\hitchin(S,\sf G_\R)$ and thus its Hessian is well defined.

\begin{corA}\label{cA}\label{c.Hitchin}For every $v\in\sf T\hitchin(S,\sf G_\R)$ and every $\sroot\in\Pi$ one has $$\Hess\Hff_\sroot(Jv)=\PP^{\sroot}(v).$$ Moreover, when $\sf G_\R=\PSL_d(\R)$ the Hessian of $\Hff_{\sroot_1}:\frak X_\Pi\big(\pi_1S,\PSL_d(\C)\big)\to\R,$ at a Hitchin point $\rho,$ is strictly positive on every direction except $\sf T_\rho\hitchin_d(S),$ where it is degenerate. In particular the Hitchin locus is an isolated minimum for $\Hff_{\sroot_1}.$
\end{corA}

The second statement follows directly from the first, together with the aforementioned non-degeneracy result by B.-Canary-Labourie-S. \cite{LiouvillePressure} for the spectral gap pressure form $\PP^{\sroot_1}.$ 

Corollary~\ref{cA} brings further evidence that the spectral gap pressure form is more geometric than the spectral radius pressure form, and shares more similarities to the classical Weil-Petersson metric. A key ingredient in its proof is the notion of $(1,1,2)$-hyperconvex representations, studied in P.-S.-W. \cite{PSW1} (see Theorem \ref{h->h}). These are representations $\rho:\G\to\PSL_d(\C)$ that are Anosov with respect to the first two simple roots and whose boundary maps satisfy an additional transversality condition (see Section \ref{s.112}). The main result of \cite{PSW1} then yields that, on the open set
$$\Anosov_{\{\sroot_1,\sroot_2\}}^\pitchfork\big(\G,\PGL_d(\C)\big)=\big\{\rho\in\frak X\big(\G,\PGL_d(\C)\big):\ (1,1,2)\textrm{-hyperconvex}\}$$
the Hausdorff dimension of the limit set $\xi^1_\rho(\bord\G)$ equals the \emph{critical exponent} $h^{\sroot_1}_\rho$ for the first root (see Section \ref{s.112} for the definition of the critical exponent) and is thus analytic. 

\begin{thmA}\label{tA} Let $\G$ be a word hyperbolic group with $\bord\G$ homeomorphic to a circle and let $\rho\in \Anosov_{\{\sroot_1,\sroot_2\}}^\pitchfork\big(\G,\PSL_d(\R)\big)$ be a regular point of the character variety $\frak X\big(\G,\PGL_d(\C)\big).$ Then for every differentiable curve $(\rho_t)_{t\in(-\eps,\eps)}\subset\frak X\big(\G,\PSL_d(\R)\big)$ with $\rho_0=\rho$ one has $$\Hess\Hff_{\sroot_1}(J\dot\rho)=\PP^{\sroot_1}(\dot\rho).$$
\end{thmA}

Thanks to the work of \cite{PSW1,PSW2}, Theorem \ref{tA}  applies not only to Hitchin components but to $\Theta$-positive Anosov representations into indefinite orthogonal groups \cite{GWpos}. However, in all these cases, we do not know for which roots $\sroot$ the associated pressure form is non-degenerate. 

\begin{corA}\label{c.hrtt}Let $\cal P\big(\G, \SO_0(p,q)\big)$ denote the set of $\Theta$-positive Anosov representations.  Then for every $v\in\sf T\cal P\big(\G, \SO_0(p,q)\big)$ and every $\sroot\in\{\sroot_1,\ldots, \sroot_{p-2}\}$ one has$$
\Hess\Hff_\sroot(Jv)=\PP^{\sroot}(v).$$
\end{corA}

The second main result of the paper is a generalization of B. \cite{martincriticos} to Anosov representations of word hyperbolic groups that are not necessarily virtual surface groups.  

\begin{thmA}\label{tB} Let $\G$ be a word hyperbolic group and let $\rho\in\frak X_{\{\sroot_1,\sroot_2\}}^\pitchfork\big(\G,\PSL_d(\C)\big)$ be a smooth point. Assume moreover that 
$$\Hff_{\sroot_1}:\Anosov_{\{\sroot_1,\sroot_2\}}^\pitchfork\big(\G,\PSL_d(\C)\big)\to\R_+$$
 is critical at $\rho.$ Then $\Hess_\rho\Hff_{\sroot_1}$ is positive semidefinite on a subspace of dimension at least half the real dimension of $\Anosov_{\{\sroot_1,\sroot_2\}}^\pitchfork\big(\G,\PSL_d(\C)\big).$
If, furthermore, the pressure form $\PP^{\sroot_1}$ is non-degenerate, then there are no local maxima.\end{thmA}

This theorem applies  in particular to all convex cocompact Kleinian groups in $\PSL_2(\C)$, a result not covered by \cite{martincriticos}. In general, admitting $(1,1,2)$-hyperconvex representations is a relatively restrictive assumption on the group $\G$, as, for example, it implies that its boundary has dimension smaller than 2. However there are many classes of subgroups of $\PSL_2(\C)$ admitting $(1,1,2)$-hyperconvex representations with Zariski dense  image in $\PSL_d(\C).$ It is not yet known if the spectral gap pressure form is non-degenerate for such representations, and we hope that Theorem \ref{tB} will encourage research in this direction.

\medskip 

The proof of both Theorems \ref{tA} and \ref{tB} relies on pluriharmonicity of length functions. The mechanism behind both proofs is relatively versatile, and can be applied in many situations. As an example we use it in Corollary \ref{cC} to prove that, at a Fuchsian representation, the critical exponent $h^{\omega_1}(\rho)$ relative to the first fundamental weight necessarily increases along purely imaginary deformations (see Section \ref{s.5.5} for details).

Our last result is another application of the previous techniques in a rank one situation. Recall that a representation $\rho:\G\to\PU(n,1)$ is convex-cocompact if and only if it is projective Anosov when $\PU(n,1)$ is considered as a subgroup of $\PGL_{n+1}(\C)$ through the standard inclusion. Moreover, for $\g\in\G$ the real length of the associated closed geodesic is the spectral radius $\omega_1\Big(\lambda\big(\rho(\g)\big)\Big).$ 

If $\rho:\G\to\PU(n,1)$ is convex co-compact then we denote by $\Hff_{\bord\C\HH^n}(\rho)$ the Hausdorff dimension the limit set of $\G$ in the visual boundary of the complex-hyperbolic space, with respect to a visual metric. A celebrated result by Sullivan \cite{sullivan} in real-hyperbolic space, and further extended to arbitrary negatively curved manifolds by Yue \cite{yue}, asserts that if $\rho:\G\to\PU(n,1)$ is convex-co-compact then $$\Hff_{\bord\C\HH^n}(\G)=h^{\omega_1}(\rho).$$

Assume now that $\G$  is a co-compact lattice of $\PSO(n,1)$ and denote by $\iota$ the inclusion $$\iota:\G\to\PSU(n,1),$$ where the inclusion of $\PSO(n,1)\to\PSU(n,1)$ is given by extending the coefficients. Bourdon \cite{bourdon2} proved that $\iota$ is a global rigid minima for $\Hff_{\bord\C\HH^n}$ among convex co-compact representations of $\G$ in $\PU(n,1).$ In section \ref{prooftC} we prove the following strengthening.

\begin{thmA}\label{thmComplex}Assume $\iota$ is a regular point of the character variety $\frak X\big(\G,\PSU(n,1)\big),$ then $\Hess_\iota\Hff_{\bord\C\HH^n}$ is positive definite in any direction not tangent to $\frak X\big(\G,\PSO(n,1)\big).$
\end{thmA}

If $n>2$ then Mostow's classical rigidity result states that $\iota$ is an isolated point of $\frak X\big(\G,\PSO(n,1)\big)$ so Theorem \ref{thmComplex} implies that the Hessian of the Hausdorff dimension at $\iota$ is positive definite.

\begin{obs}
If $\rho$ is a small deformation of $\iota$ in $\PSU(n,1)$ then it is a convex co-compact subgroup of $\PSU(n,1)$ and moreover, by P.-S.-W. \cite[Corollary 8.5]{PSW1}, the Hausdorff dimension of the limit set for a \emph{Riemannian metric} on $\bord\C\HH^n$ also coincides with $h^{\omega_1}(\rho),$ so in Theorem \ref{thmComplex} we can consider $\Hff_{\bord\C\HH^n}$ either as the Hausdorff dimension for a visual metric or for a Riemannian metric. \end{obs}

\subsection*{Outline of the paper}
In Section \ref{s.2} we discuss the background on Anosov representations needed in the paper: after reviewing the  basic definitions we discuss, in Section \ref{s.112}, the results of \cite{PSW1} which singled out   $(1,1,2)$-hyperconvex representations. In Section \ref{s.hrtt} we recall the basic facts about higher rank Teichm\"uller theories needed to deduce Corollaries \ref{cA} and \ref{c.hrtt} from Theorem \ref{tA}; finally  in Section \ref{s.rep} we discuss an important  dynamical viewpoint on Anosov representations:  these can be thought of as reparametrizations of the geodesic flow of $\G$, and  are thus amenable to the thermodynamic formalism. In Section \ref{s.3} we discuss the thermodynamic formalism needed to define, in Section \ref{b2}, the pressure forms. We conclude the paper in Section \ref{s.proofs} introducing the main technical tool of the paper, pluriharmonicity of dynamical intersection, which is directly used to prove Theorems \ref{tA}, \ref{tB} and \ref{thmComplex}.

\section{Anosov representations}\label{s.2}
In this section we introduce the necessary background on Anosov representations, and recall how they give rise to reparametrizations of the geodesic flow. 
\subsection{Basic notions}
We recall the Cartan and the Jordan-Lyapunov projections and the characterization of Anosov representations we are going to use. 

Let $\sf G$ be a semisimple real-algebraic Lie group of non-compact type with finite center, for $\K=\R$ or $\C$ denote by $\sf G_\K$ the group of the $\K$-points of $\sf G.$ Fix a maximal compact subgroup $\sf K<\sf G_\K$ with Lie algebra $\frak t$. We denote by $\sf E<\frak t^\bot$ a Cartan subalgebra, and by $\Delta\subset \sf E^*$ a choice of simple roots. This corresponds to the choice of a Weyl chamber in $\sf E$, which we will denote by $\sf E^+$. In the case $\sf G=\PSL_d$ we identify $\sf E^+$ with $$\big\{(x_1,\ldots,x_d)\in \R^d| x_1\geq \ldots \geq x_d, \sum x_i=0\big\}.$$ 

Every element $g\in\sf G_\K$ can be written as a product $$g=k_1\exp\big(\sigma(g)\big) k_2$$ 
for $k_1,k_2\in\sf K$ and  a unique element $\sigma(g)\in \sf E^+$, the \emph{Cartan projection} of $g$. If $\sf G_\R=\PSL_d(\R)$, the numbers $\sigma_i(g)$ are the logarithms of the square roots of the eigenvalues of the symmetric matrix $g^tg$. If $\sroot\in\Delta$ then we denote by $\omega_\sroot$ its associated fundamental weight. 

Let $\Theta \subset \Delta$ be a subset of simple roots. We denote by $\sf P_\Theta<\sf G$ the associated parabolic subgroup, by $\wk{\sf P}_\Theta$ the opposite associated parabolic group and by
$$\sf E_\Theta:=\bigcap_{\sroot\in\Delta\setminus \Theta}\ker(\sroot )$$ the Lie algebra of the center of the Levi group $\sf P_\Theta\cap\wk{\sf P}_\Theta.$ It comes equipped with the natural projection \begin{equation}\label{projection}p_\Theta:\sf E\to\sf E_\Theta\end{equation} parallel to $\bigcap_{\sroot\in\Theta}\ker\sroot.$ Finally let $\sf E_\Theta^*<\sf E^*$ be the subspace generated by the fundamental weights associated to elements in $\Theta$
 $$\sf E_\Theta^*:=\langle\omega_\sroot|\;\sroot\in\Theta\rangle=\{\varphi\in\sf E^*|\;\varphi\circ p_\Theta=\varphi\}.$$ One has that $\sf E_\Delta=\sf E$ and $\sf P_\Delta$ is a minimal parabolic subgroup.

Let $\G$ be a finitely generated discrete group and denote by $|\,|$ the word length for a fixed finite symmetric generating set. 

\begin{defi}\label{defAnosov} Let $\Theta\subset\Delta$. A representation $\rho:\G\to\sf G_\K$ is \emph{$\Theta$-Anosov} if there exist positive constants $c,\mu$ such that for all $\g\in\G$ and $\sroot\in\Theta$ one has 
\begin{equation}\label{def1}
\sroot\Big(\sigma\big(\rho(\g)\big)\Big)\geq \mu|\g|-c.
\end{equation} 
A $\{\sroot_1\}$-Anosov representation $\rho:\G\to\PGL_d(\K)$ will be called \emph{projective Anosov}.
\end{defi}
Note that this is not the original definition given in Labourie and Guichard-W. \cite{Labourie-Anosov, GW-Domains}, but a characterization due to Kapovich-Leeb-Porti and Bochi-Potrie-S. \cite{KLP-Morse,BPS}. Note also that there is a recent characterization by Kassel-Potrie \cite{KP} only in terms of the Jordan-Lyapunov projection (see below for the definition) rather than the Cartan projection.

Anosov representations are quasi-isometric embeddings, thus in particular they are injective and have discrete image. It was proven in \cite{KLP-Morse} (see also \cite{BPS}) that only  word hyperbolic groups admit Anosov representations; we will denote by $\bord \G$ the Gromov boundary of the group $\G$. 

A key property of Anosov representations is the existence of equivariant boundary maps with good dynamical properties \cite{Labourie-Anosov,  GW-Domains, GGKW, KLP-Morse, BPS}. With our definition, the existence of boundary maps for such representations is a Theorem of \cite{KLP-Morse} and \cite{BPS}. From now on we will restrict ourselves, without loss of generality, to self-opposite subsets $\Theta\subset \Delta$ 

 \begin{thm}[Kapovich-Leeb-Porti \cite{KLP-Morse}]
Let $\rho:\G\to\sf G_\K$ be $\Theta$-Anosov. Then there exist a unique dynamics preserving, continuous, transverse equivariant boundary map 
$$\xi_\rho^{\Theta}:\bord\G\to \sf G_\K/\sf P_\Theta.$$ 
\end{thm}
If $\sf G=\PGL_d$ and $\Theta=\{\sroot_r\}$, then $\sf G_\K/\sf P_\Theta=\Gr_r(\K^d)$ and we write $\xi^r_\rho=\xi_\rho^{\{\sroot_r\}}.$ 

It was proven in \cite{GW-Domains} that it is possible to reduce the study of general $\{\sroot\}$-Anosov representations to projective Anosov representations.  Indeed one can use the following result by Tits, since for the representations $\Wedge_\sroot$ below one has $$\sroot_1\Big(\sigma\big(\Wedge_\sroot(\rho(\g))\big)\Big)=\sroot\Big(\sigma\big(\rho(\g)\big)\Big).$$

 \begin{prop}[Tits {\cite{Tits}}]\label{TitsReps}
For every $\sroot\in\Delta$ there exists an irreducible proximal representation $\Wedge_\sroot:\sf G_\R\to \PGL_d(\R)$ whose highest restricted weight is $l\omega_\sroot$ for some $l\in\N$.
\end{prop}

%Let us simply recall that the fact that the highest weight of the representation $\Wedge_\sroot$ is $l\omega_\sroot$ is equivalent to $\forall g\in\sf G_\R$ one has  \begin{equation}\label{highestWeight}\lambda_1\big(\Wedge_\sroot(g)\big)=l\omega_\sroot\big(\lambda(g)\big)\end{equation}

Recall that the Jordan decomposition states that every $g\in\sf G_\K$ can be written uniquely as a commuting product $g=g_eg_hg_n,$ where $g_e$ is elliptic, $g_h$ is $\R$-split  and $g_n$ is unipotent. The \emph{Jordan-Lyapunov} projection $\lambda:\sf G_\K\to\sf E^+$ is defined by the logarithm of the  eigenvalues of $g_h$ with multiplicities and in decreasing order. If $\sf G=\PGL_d$, this corresponds to the logarithm of the modulus of the roots of the characteristic polynomial of $g$ with multiplicities and in decreasing order, and we denote by $$\lambda(g)=\big(\lambda_1(g),\ldots,\lambda_d(g)\big)\in\big\{(x_1,\ldots,x_d)\in \R^d|\; x_1\geq \ldots \geq x_d, \sum x_i=0\big\}$$ its coordinates.

We will denote by $\Lambda_\rho\subset \sf E^+$ the \emph{limit cone} of the subgroup $\rho(\G)<\sf G_\K$. This is the cone given by
$$\Lambda_\rho:=\overline{\{\R_+\cdot\lambda(\rho(\g))|\;\g\in\G\}}.$$
It was proven by Benoist \cite{limite} that, provided $\rho(\G)$ is Zariski dense,   $\Lambda_\rho$ is convex and has non-empty interior. 
 
For every functional $\varphi\in\sf E^*$ that is positive on the limit cone $\Lambda_\rho$ we denote by $h^\varphi(\rho)$ the \emph{critical exponent} of the Dirichlet series
$$s\mapsto\sum_{\g\in\G}e^{-s\varphi\Big(\sigma\big(\rho(\g)\big)\Big)},$$ it can be computed as $$h^\varphi(\rho)=\inf\Big\{s:\sum_{\g\in\G}e^{-s\varphi\big(\sigma\big(\rho(\g)\big)\big)}<\infty\Big\}=\sup\Big\{s:\sum_{\g\in\G}e^{-s\varphi\big(\sigma\big(\rho(\g)\big)\big)}=\infty\Big\}.$$

\subsection{Hyperconvex representations}\label{s.112}

We begin with the following definition from \cite{PSW1}:
\begin{defi*} A $\{\sroot_1,\sroot_2\}$-Anosov representation $\rho:\G\to\PGL_d(\K)$ is called $(1,1,2)$-\emph{hyperconvex} if for every triple of pairwise distinct points $x,y,z\in\bord\G$ one has $$\big(\xi^1_\rho(x)\oplus \xi^1_\rho(y)\big)\cap \xi^{d-2}_\rho(z)=\{0\}.$$ \end{defi*}

The following is a direct consequence of the uniqueness of boundary maps:

\begin{lemma}\label{FieldExt} The complexification of a real hyperconvex representation is hyperconvex (over $\C$).
\end{lemma}

An important property of (1,1,2)-hyperconvex representations, established in \cite{PSW1} is that, for these representations, the Hausdorff dimension of the limit curve for a Riemannian metric on $\P(\K^d)$ computes the critical exponent for the first simple root. If $\rho$ is $\{\sroot_1\}$-Anosov, the root $\sroot_1$ is positive on the limit cone (recall Equation \eqref{def1}) and thus its critical exponent is well defined. We then have the following.

\begin{thm}[{P.-S.-W. \cite{PSW1}}]\label{t.LC} Let $\rho:\G\to\PGL_d(\K)$ be $(1,1,2)$-hyperconvex, then $$\dim_{\Hff}\big(\xi^1(\bord\G)\big)=h^{\sroot_1}_\rho.$$ 
\end{thm}

A second important property of $(1,1,2)$-hyperconvex representations into $\PSL_d(\R)$ was established in P.-S.-W. \cite{PSW1} (and independently in Zimmer-Zhang \cite{ZZ}) is the following: if $\G$ is such that $\bord\G$ is homeomorphic to a circle,  then the image of the boundary map $\xi^1_\rho$ is a $\textrm{C}^1$-curve. As a result we get

\begin{thm}[{P.-S.-W. \cite{PSW1}}]\label{h=1}
Let $\rho:\G\to\PSL_d(\R)$ be $(1,1,2)$-hyperconvex. If $\bord\G$ is homeomorphic to a circle then $\dim_{\Hff}\big(\xi^1(\bord\G)\big)=1.$
\end{thm}

Thus, for fundamental groups of surfaces, the Hausdorff dimension is constant and minimal on the real (1,1,2)-hyperconvex locus.

\subsection{Higher rank Teichm\"uller Theory}\label{s.hrtt}
A higher Teichm\"uller space is a union of connected components of a character variety $\frak X(\pi_1S,\sf G_\R)$ only consisting of Anosov representations. 

Historically, the first family of higher Teichm\"uller spaces are Hitchin components. They arise whenever $\sf G_\R$ is a center free real-split simple Lie group. In this case there is a unique \emph{principal} subalgebra $\frak{sl}_2(\R)<\frak g_\R$ characterized by the property that the centralizer of $\left(\begin{smallmatrix}0&1\\0&0\end{smallmatrix}\right)$ has minimal dimension (Kostant \cite{kostant}). The \emph{Hitchin component} $\hitchin(S,\sf G_\R)\subset \frak X(\pi_1S,\sf G_\R)$ is the connected component containing \emph{Fuchsian representations}: the composition of the holonomy of a hyperbolization $\pi_1S\to\PSL_2(\R)$ and the morphism $\PSL_2(\R)\to \sf G_\R$ induced by the inclusion of the principal subalgebra. Representations in the Hitchin component are Anosov with respect to the minimal parabolic subgroup \cite{Labourie-Anosov}.\footnote{In fact, Labourie proved this for $\sf G_\R = \PSL_d(\R)$, which implies the result also for symplectic groups and odd-dimensional orthogonal groups. Fock and Goncharov gave a  characterization of representations in the Hitchin component as positive representations, from which the Anosov property can be deduced with a little work.}
Furthermore, representations in the Hitchin component are hyperconvex:
\begin{thm}[{\cite{Labourie-Anosov,PSW1,S}}]\label{h->h}Let $\sf G_\R$ be  a simple split center-free real group. For every $\rho\in\hitchin(S,\sf G_\R)$ and $\sroot\in\Pi$ the representation $\Wedge_\sroot\rho:\pi_1S\to\PSL(V,\R)$ is $(1,1,2)$-hyperconvex.
\end{thm}
\begin{proof}
This was established, for the groups $\sf G_\R=\PSL_d(\R),$ $\PSp(2n,\R),$ $\PSO(n,n+1)$ or the split form of the exceptional complex Lie group $\sf G_2,$ by Labourie \cite{Labourie-Anosov}, for $\sf G=\SO(n,n)$ by P.-S.-W. \cite[Theorem 9.9]{PSW1}. The general case follows from S. \cite[Remark 5.14]{S}.
\end{proof}

The second family of higher Teichm\"uller spaces are spaces of  maximal representations in Hermitian Lie groups $\sf G_\R$ \cite{MaxReps}. Our results here do not apply in this setting. Maximal representations are, in general, only Anosov with respect to one root $\sroot$, which therefore doesn't belong to the Levi-Anosov subspace. Even though we know that for maximal representations the critical exponent $h^\sroot_\rho$ is constant and equal to one (\cite[Theorem 1.2]{PSW2}), it is not clear if a spectral gap pressure metric $\PP^\sroot$ can be constructed in this case. Moreover, since maximal representations are, in general, not (1,1,2)-hyperconvex, it is not known if, for complex deformations in $\rho:\G\to\sf G_\C$, the critical exponent $h^\sroot_\rho$ equals the Hausdorff dimension of the limit set.

Conjecturally there are two further families of higher Teichm\"uller spaces, given by $\Theta$-positive representations as introduced in \cite{GWpos, ICM}. $\Theta$-positive representations exist when $\sf G_\R$ is locally isomorphic to $\SO(p,q)$, $p<q$, or when $\sf G_\R$ belongs to a special family of exceptional Lie groups. In a forthcoming article, Guichard, Labourie and W. \cite{GLWpos} prove that $\Theta$-positive representations are $\Theta$-Anosov; in particular, in the case of $\SO(p,q)$, $p<q$, $\Theta$-positive representations are Anosov with respect to the first $p-1$ roots. Since this article is not yet available, we will here consider $\Theta$-positive Anosov representations, and use the following result from \cite{PSW2}. 

\begin{thm}[{\cite[Theorem 10.1]{PSW2}}]
Let $\rho:\G\to \SO(p,q)$ be a $\Theta$-positive $\Theta$-Anosov representation. For every $\sroot\in\{\sroot_1,\ldots,\sroot_{p-2}\}$ the representation $\Wedge_\sroot\rho:\pi_1S\to\PSL(V,\R)$ is $(1,1,2)$-hyperconvex.
\end{thm}

Note that when $\sf G_\R$ admits a $\Theta$-positive structure, Guichard and W. conjectured several years ago, see also \cite{GWpos, ICM}, that then there exist additional connected components (namely the conjectured components of $\Theta$-positive representations) in the representation variety, which are not distinguished by characteristic numbers. This conjecture has been proven by Collier \cite{Collier} in the case of $\sf G_\R = \SO(n,n+1)$ and by Aparicio-Arroyo, Bradlow, Collier, García-Prada, Gothen, and Oliveira \cite{ABCGGO} in the case of $\sf G_\R = \SO(p,q)$ using methods from the theory of Higgs bundles. 

\subsection{Reparametrizations of geodesic flows}\label{s.rep}
In this section we describe a very useful dynamical viewpoint on Anosov representations from S. \cite{quantitative} and B.-Canary-Labourie-S. \cite{pressure}, which makes them amenable to the thermodynamic formalism: Any Anosov representation gives rise to a repa\-ra\-me\-tri\-za\-tion of the geodesic flow. 

Given a hyperbolic group $\G$ we denote by $\sf U\G$ the Gromov geodesic flow; this is a metric space  endowed with a topologically transitive flow $\phi$ whose periodic orbits correspond to conjugacy classes in $\G$. If $\G$ admits an Anosov representation then $\phi$ is moreover metric Anosov \cite{pressure}. Note that, if $\G$ is the fundamental group of a compact negatively curved manifold $M$, we can choose $\sf U\G=\sf U M$; more generally,  whenever $\G$ admits an Anosov representation, its geodesic flow can be explicitly constructed  with the aid of the associated boundary maps \cite[Theorem 1.10]{pressure}.

If $\alpha>0$, we denote by $\holder_\alpha(\sf U\G,\R)$ the space of $\alpha$-H\"older continuous functions on $\sf U\G$ and by $\holder(\sf U\G,\R)$ is the space of all H\"older continuous functions. If $f \in \holder (\sf U\G,\R)$ and $a \in [\G]$ is a conjugacy class, then we define the $f$-period of $a$ by
$$\ell_f(a) = \int_0^{\ell(a)} f(\phi_t(x))dt$$
where $x \in a$. 
If $f \in \holder_\alpha(\sf U\G,\R_+)$, we obtain a new flow $\phi^f$ on $\sf U\G$ called the {\em reparametrization} of $\phi$ by $f$. The flow $\phi^f$ is given by the formula
\begin{equation}\label{e.rep}\phi_t(x)= \phi^f_{k_f(x,t)}(x) \end{equation}
where $k_f(x,t) = \int_0^t f(\phi_sx)ds$ 
for all $x\in X$ and $t\in\R$. The flow $\phi^f$
is H\"older orbit equivalent to $\phi$ and if $a\in [\G]$, then $\ell_f(a)$ is the period
of $a$ in the flow $\phi^f$. 

In \cite[Section 4]{pressure} B.-Canary-Labourie-S. associate to any projective Anosov representation a reparametrization of the geodesic flow $\sf U\G$. They prove the following statement,  the second part is proved in \cite[Section 6]{pressure}.

\begin{prop}[{\cite{pressure}}]\label{spectralPotential} Let $\rho:\G\to\PGL_d(\R)$ be a projective Anosov representation. Then there exists a positive H\"older-continuous function $f_\rho^{\lambda_1}:\UG\to\R_{>0}$ such that for every conjugacy class $[\g]\in[\G]$ one has $$\ell_\g(f_\rho^{\lambda_1})=\lambda_1\big(\rho(\g)\big).$$ Moreover, if $\{\rho_u\}_{u\in D}$ is an analytic family of such representations, then one can choose $f_{\rho_u}^{\omega_1}$ so that the function $u\mapsto f_{\rho_u}^{\lambda_1}$ is analytic. 
\end{prop}

Proposition \ref{spectralPotential} together with Tits Proposition \ref{TitsReps} directly give the following from Potrie-S. \cite{exponentecritico}, where $\K=\R$ case is treated. When $\K=\C$ the result follows at once by considering $\sf G_\C$ as a real group. Recall equation (\ref{projection}) for the definition of $p_\Theta:\sf E\to\sf E_\Theta.$

\begin{cor}[{\cite[Cor. 4.5]{exponentecritico}}]\label{potential}\label{potentialC} Let $\rho:\G\to\sf G_\K$ be $\Theta$-Anosov, then there exists a positive H\"older-continuous function $f_\rho^{\Theta}:\UG\to\sf E_\Theta$ such that for every conjugacy class $[\g]\in[\G]$ one has $$\ell_\g(f_\rho^{\Theta})=p_\Theta\Big(\lambda\big(\rho(\g)\big)\Big).$$ Moreover, if $\{\rho_u\}_{u\in D}$ is an analytic family of such representations, then one can choose $f_{\rho_u}^{\Theta}$ so that the function $u\mapsto f_{\rho_u}^{\Theta}$ is analytic. 
\end{cor}

Thus, Corollary \ref{potentialC} readily implies that if $\rho$ is $\Theta$-Anosov then for every $\varphi\in(\sf E_\Theta)^*$ that is strictly positive on $\Lambda_\rho-\{0\}$ there exists a reparametrization of the geodesic flow of $\G$ whose periods are given by\footnote{Recall that for every $\varphi\in(\sf E_\Theta)^*$ one has $\varphi\circ p_\Theta=\varphi.$} $$\varphi\big(\lambda(\rho(\g))\big).$$ namely, if we denote by $f^\varphi_\rho=\varphi(f^\Theta_\rho)$ then one considers the flow $\phi^{f_\rho^\varphi}.$ We will need in the following that, in this situation, the critical exponent $h^\varphi(\rho)$ is also the entropy of the flow $\phi^{f^{\varphi}_\rho}.$ This can be found for example in Ledrappier \cite{ledrappier}, S. \cite{quantitative} and on Glorieux-Monclair-Tholozan \cite{GMT} for the general version.

\begin{prop}
Let $\rho:\G\to \sf G_\K$ be $\Theta$-Anosov. For each $\varphi\in \sf E_\Theta^*$ strictly positive on $\Lambda_\rho-\{0\}$ it holds that 
$$h^\varphi(\rho)=\lim_{T\to\infty}\frac {\log \#\Big\{\g\in[\G]|\;\varphi\Big(\lambda\big(\rho(\g)\big)\Big)<T\Big\}}{T}$$
\end{prop}
This applies, in particular, to the root $\sroot_1$ if a representation $\rho$ is (1,1,2)-hyper\-con\-vex.
\section{Thermodynamic formalism} \label{s.3}
We now briefly describe the thermodynamic formalism introduced by Bowen, Ruelle, Parry, Pollicott (among others), and in particular the pressure function on the space of H\"older observables on a metric space with a H\"older flow (see \cite{ruelle}). This will then be used, in Section \ref{b2},  to define various pressure forms $\PP^\varphi$ on  subsets of the representation variety $\frak X(\G,\PSL_d(\R))$ by assigning to each representation $\rho$ the H\"older function $f_\rho^\varphi$ on the geodesic flow space $\UG$ of the group. 

For a moment we forget about representations and let $X$  be a compact metric space  with a H\"older continuous flow \hbox{$\phi=\{\phi_t:X \rightarrow X\}_{t \in \R}$}   without fixed points.  We denote by  $O$  the collection of periodic orbits of the flow $\phi$. For $a \in O$, we let $\ell(a)$ be the length of the periodic orbit $a$.

As in Section \ref{s.rep} we denote by  $\holder_\alpha(X,\R)$ the space of $\alpha$-H\"older continuous functions on $X$ for some $\alpha>0$, and we set the $f$-period of $a \in O$ to be
$$\ell_f(a) = \int_0^{\ell(a)} f(\phi_t(x))dt.$$
 Two maps $f,g\in \holder_\alpha(X,\R)$ are called {\em Liv\v sic cohomologuous} if there exists $U:X\to\R$ 
such that, for all $x\in X$, then 
$$f(x)-g(x)=\left.\frac{\partial}{\partial t}\right|_{t=0}U(\phi_t x).$$
It follows that if  $f$ and $g$ are Liv\v sic cohomologous then $\ell_f(a)=\ell_g(a)$ for all $a\in O$. If $f \in \holder_\alpha(X,\R_+)$, we denote by $\phi^f$ the {\em reparametrization} of $\phi$ by $f$, which is the flow on $X$ defined by \eqref{e.rep}.

We let $\mathcal M_\phi$ be the set of $\phi$-invariant probability measures on $X$. In particular if $\delta_a$ is the Lebesgue measure on the periodic orbit $a$, then $\hat\delta_a = \delta_a/\ell(a) \in \mathcal M_\phi$.  For $\mu \in \mathcal M_\phi$ we denote by $h(\phi,\mu)$ its metric entropy. Then, for $f \in \holder_\alpha(X,\R)$, the {\em topological pressure} is
$$P(f)=\sup_{m\in \mathcal M_\phi}\left\{h(\phi,m)+\int_X f dm\right\}.$$
Note that the topological pressure $P$ depends on the flow $\phi$, but we will omit this in the notation.
The {\em topological entropy}  of a flow is given by $h_{\mathrm{top}}(\phi) = P_\phi(0)$. A measure $m_f$ that attains this supremum is called an {\em equilibrium state} for $f$ and an equilibrium state for the zero function is called a {\em measure of maximal entropy}.  

We note that $P(f)$ only depends on the Liv\v sic cohomology class of $f$.

\begin{lemma}[{S. \cite[Lemma 2.4]{quantitative}}] Let $\phi$ be a H\"older continuous flow on a compact metric space X and $f\in \holder_\alpha(X,\R_+)$. Then
$$P(-hf) = 0$$
if and only if $h = h_{\mathrm{top}}(\phi^f)$. Moreover, if  m is an equilibrium state of $-h_{\mathrm{top}}(\phi)f$, then $fm$ is a positive multiple of a
measure of maximal entropy for the flow $\phi^f$.
\end{lemma}

We now restrict to  {\em  transitive metric Anosov flows}. In the manifold setting a metric Anosov flow $\phi$ corresponds to a standard  Anosov flow  where the unit tangent bundle of $X$ has a $\phi$-invariant decomposition $T_1(X) = E_-\oplus E_0 \oplus E_+$ where $E_-$ is contracting under the  flow, $E_0$ is the direction of the flow and $E_+$ is contracting under the flow reverse flow of $\phi$ (see \cite{smaleflows} for details). We have the following theorem of Liv\v sic.

\begin{thm}[{Liv\v sic's Theorem, \cite{livsic}}]
Let $\phi$ be a  transitive metric Anosov flow.  If $f \in \holder_\alpha(X,\R)$ then $\ell_f(a) = 0$ for all $a\in O$ if and only if $f$ is Liv\v sic cohomologous to $0$.
\end{thm}

It follows that for metric Anosov flows, the Liv\v sic cohomology class of $f$ is determined by its periods.

Given $f \in \holder_\alpha(X,\R)$ we let
$$R_T(f) = \{ a\in O\ | \ \ell_f(a) \le T\}.$$
Then we have the following;
\begin{thm}[{Bowen \cite{Bowen1}, Bowen-Ruelle \cite{bowenruelle}, Pollicott \cite{smaleflows}}]
Let $\phi$ be a  transitive metric Anosov flow and $f \in \holder_\alpha(X,\R_+)$ nowhere vanishing. Then
$$h(f) = \lim_{T\rightarrow \infty} \frac{\log\#R_T(f)}{T} = h_{\mathrm{top}}(\phi^f)$$
is finite and positive. Moreover for all $g  \in \holder_\alpha(X,\R)$ there exists a unique equilibrium state $m_g$ for $g$. The measure of maximal entropy  $\mu_\phi$ for the flow $\phi$ is
$$\mu_\phi =  \lim_{T\rightarrow \infty} \frac{1}{\#R_T(1)}\sum_{a \in R_T(1)} \frac{\delta_a}{\ell(a)}.$$
 \end{thm}
 
 Furthermore for Anosov flows the derivatives of the Pressure function satisfy the following.
 
 \begin{prop}[{Parry-Pollicott \cite{parrypollicott}}]
 Let $\phi$ be a  transitive metric Anosov flow and $f,g \in \holder_\alpha(X,\R)$. Then
 \begin{enumerate}
 \item The function $t\rightarrow P(f+tg)$ is analytic
 \item The first derivative satisfies
 $$\left.\frac{\partial P(f+tg)}{\partial t}\right|_{t=0} = \int gdm_f, $$ where $m_f$ is the equilibrium state for $f$. 
\item If $\int g dm_f = 0$ (mean-zero) then
 $$\left. \frac{\partial^2 P(f+tg)}{\partial t^2}\right|_{t=0} = \lim_{T\rightarrow \infty}\int\left(\int_0^Tg(\phi_s(x))ds\right)^2 dm_f(x) = \Var(g,m_f).$$
\item If $\Var(g,m_f) = 0$ then $g$ is Liv\v sic cohomologous to zero.\end{enumerate}
 \label{pderiv}
 \end{prop}
 
Using the above, in \cite{mcMWP} McMullen defined the Pressure semi-norm as follows. We let $\mathcal P(X)$ be the space of pressure zero functions, i.e.
$$\mathcal P(X) = \{ F \in \holder(X,\R)\ |\  P(F) = 0\}.$$
Then by Proposition \ref{pderiv}(ii), the tangent space to $\mathcal P(X)$ at $F$ can be identified with 
$$\sf T_F(\mathcal P(X)) = \left\{ g \in \holder(X,\R)\ | \int g dm_F = 0\right\},$$ where  $m_F$ is the equilibrium state for $F$.
Then the {\em pressure semi-norm} of $g \in \sf T_F(\mathcal P(X)) $ is 
$$\PP(g) = -\frac{\Var(g,m_F)}{\int F dm_F}.$$
By Proposition \ref{pderiv} it follows that  $\PP(g)$ only depends on the Liv\v sic-cohomology class $[g]$ and is positive definite in the sense that it is zero if and only if $[g] = 0$. Therefore it can be considered as a (positive-definite) metric on the space of Liv\v sic cohomology classes.

The {\em dynamical intersection} is defined in \cite{pressure} as follows; if $f,g\in \holder_\alpha(X,\R)$ are positive, then their dynamical intersection is 
\begin{equation}
\label{II}
\II(f,g)=\lim_{T\to\infty} \frac{1}{\#R_T(f)}\sum_{a\in R_T(f)} \frac{\ell_{g}(a)}{\ell_{f}(a)} =\frac{\int g {\rm d}m_{-h_ff}}{\int f {\rm d}m_{-h_ff}}.
\end{equation}
The last equality follows from \cite[Sec. 3.4]{pressure}. Similar definitions have been studied in different situations, for example by Bonahon \cite{B-intersection}, Burger \cite{burger} and Knieper \cite{kni95}.

The {\em renormalized dynamical intersection} is $$\JJ(f,g):=\frac{h(g)}{h(f)}\II(f,g).$$

\begin{prop}[{B.-Canary-Labourie-S. \cite[Proposition 3.8]{pressure}}]\label{JCritical} For every pair of positive H\"older-continuous functions $f$ and $g$ one has $\JJ(f,g)\geq1.$ In particular $\JJ(f,\cdot)$ is critical at $f$ which gives 
\begin{equation}\label{dot h}
\left.\frac{\partial}{\partial t}\right|_{t=0} \log h(f_t)=\left.\frac{\partial}{\partial t}\right|_{t=0}\II(f,f_t),
\end{equation}
where $(f_t)_{t\in(-\eps,\eps)}$ is a $\class^1$ curve of positive H\"older-continuous functions with $f_0 = f$.\end{prop}

Then we have:

\begin{thm}[{B.-Canary-Labourie-S. \cite{pressure}}]\label{HessJ}
Let $\phi$ be a  transitive metric Anosov flow  on
a compact metric space $X$. If $f_t \in \holder(X,\R_+), t \in (-1,1)$ is a $1$-parameter family and $F_t = -h_{f_t} f_t$ then 
$$\left. \frac{\partial^2 }{\partial t^2}\right|_{t=0} \JJ(f_0,f_t) = \PP\big(\dot F_0\big)$$
\end{thm}

The following proposition characterizes degenerate vectors for the second derivative of $\JJ.$

\begin{prop}[{B.-Canary-Labourie-S. \cite[Lemma 9.3]{pressure}}]\label{0} Let $(f_t)_{t\in(-\eps,\eps)}$ be a $\class^1$ curve of positive H\"older-continuous functions. Then $(\partial^2/\partial t^2)|_{t=0}\JJ(f_0,f_t)=0$ if and only if for every periodic orbit $\tau$ one has $$\left.\frac{\partial}{\partial t}\right|_{t=0}h(f_t)\ell_\tau(f_t)=0.$$
\end{prop}

\section{Pressure forms}\label{b2} 
Now we will apply the thermodynamic formalism to representations. For this we make use of the interpretation of a $\Theta$-Anosov representation as a reparametrization of the geodesic flow as explained in Section~\ref{s.rep}. 

Given any functional $\varphi\in \sf E_\Theta^*$ that is positive on the limit cone, one can associate a reparametrization $f^\varphi_\rho$ of the geodesic flow on $\G$. Here we describe in detail two special cases of this construction which play an important role in the paper

\subsection{Spectral radius pressure form}
Let $\rho,\eta$ be two projective Anosov representations (with possibly different target groups). They both give rise to reparametrizations of the geodesic flow $f_\rho^{\omega_1}$ and $f_\eta^{\omega_1}$, where $\omega_1$ is the first fundamental weight. 

We define the  \emph{spectral radius dynamical intersection} of the two projective-Anosov representations $\rho,\eta$ to be the dynamical intersection between  $f_\rho^{\omega_1}$ and $f_\eta^{\omega_1}$: 
$$\II^{\omega_1}(\rho,\eta)=\II(f_\rho^{\omega_1},f_\eta^{\omega_1}).$$ 
Analogously we define  $\JJ^{\omega_1}(\rho,\eta).$ 
Moreover, given a $\class^1$ curve $(\rho_t)_{t\in(-\eps,\eps)}$ of projective Anosov representations the \emph{spectral radius pressure norm} of $\dot\rho_0$ is defined by 
$$\PP^{\omega_1}_\rho(\dot\rho_0)=\left.\frac{\partial^2}{\partial t^2}\right|_{t=0}\JJ^{\omega_1}(\rho_0,\rho_t)\geq 0.$$

The spectral radius pressure norm induces a positive semidefinite symmetric bilinear two form at the smooth points of $\{\sroot_1\}$-Anosov representations. However positive semi-definiteness is as far as thermodynamics goes, and one needs geometric arguments to establish non-degeneracy. 
In \cite{pressure} B.-Canary-Labourie-S. prove non-degeneracy under some mild assumptions, giving 

\begin{thm}[{B.-Canary-Labourie-S. \cite[Theorem 1.4]{pressure}}]\label{t.pressure1} Let $\G$ be word hyperbolic, and $\sf G_\R<\PGL_d(\R)$ be reductive. The spectral radius pressure form is an analytic Riemannian metric on the space $\cal C_g(\G,\sf G_\R)$ of conjugacy classes of $\sf G_\R$-generic, regular, irreducible, projective Anosov representations.
 \end{thm}
Recall that a representation $\rho:\G\to \sf G_\R$ is $\sf G_\R$-generic if its Zariski closure contains elements whose centralizer is a maximal torus in $\sf G_\R$, and it is regular if it is a smooth point of the algebraic variety $\Hom(\G,\sf G_\R)$.

\subsection{Spectral gap pressure form}
We now consider two $\{\sroot_1,\sroot_2\}$-Anosov representations $\rho,\eta$ (with possibly different target groups). As explained in Section~\ref{s.rep} they define reparametrizations $f_\rho^{\sroot_1}$ and $f_\eta^{\sroot_1}$ of the geodesic flow. 

We define the \emph{spectral gap dynamical intersection} of $\rho$ and $\eta$ to be the dynamical intersection between  $f_\rho^{\sroot_1}$ and $f_\eta^{\sroot_1}$: $$\II^{\sroot_1}(\rho,\eta)=\II(f_\rho^{\sroot_1},f_\eta^{\sroot_1}),$$ and analogously for $\JJ^{\sroot_1}(\rho,\eta).$ 
Given a $\class^1$ curve $(\rho_t)_{t\in(-\eps,\eps)}$ of such $\{\sroot_1,\sroot_2\}$-representations the \emph{spectral gap pressure norm} of $\dot\rho_0$ is defined by $$\PP^{\sroot_1}_\rho(\dot\rho_0)=\left.\frac{\partial^2}{\partial t^2}\right|_{t=0}\JJ^{\sroot_1}(\rho_0,\rho_t)\geq0.$$

The spectral gap pressure norm induces a semidefinite symmetric bilinear two form on smooth points of $\{\sroot_1,\sroot_2\}$-Anosov representations. This looks very similar to the spectral radius pressure norm. It is, however, in general harder to check when the spectral gap pressure form is non-degenerate. As far as the authors know this has, so far, only been established for the Hitchin component in $\PSL_d(\R)$:

\begin{thm}[{B.-Canary-Labourie-S. \cite[Theorem 1.6]{LiouvillePressure}}]\label{gap>0Hitchin} Let $\sf G_\R$ denote either $\PSL_d(\R),$ $\PSp(2n,\R),$ $\PSO(n,n+1)$ or the split form of the exceptional complex Lie group $\sf G_2.$ Then the spectral gap pressure form is positive definite on the Hitchin component $\hitchin(S,\sf G_\R).$\end{thm}

\subsection{Vanishing directions}

Complex conjugation of matrices is an external automorphism of $\PSL_d(\C)$ and thus induces an involution 
$$\conj:\frak X\big(\G,\PSL_d(\C)\big)\to\frak X\big(\G,\PSL_d(\C)\big)$$
 whose fixed point set contains $\frak X\big(\G,\PSL_d(\R)\big).$ If $\rho\in\frak X\big(\G,\PSL_d(\R)\big)$ is a regular point, then the differential $d_\rho\conj$ splits the tangent space as a sum of \emph{purely imaginary vectors} and the tangent space to the real characters: $${\sf{T}}_\rho\frak X\big(\G,\PSL_d(\C)\big)=\Bending\oplus{\sf{T}}_\rho\frak X\big(\G,\PSL_d(\R));$$ 
the almost complex structure $J$ of $\frak X\big(\G,\PSL_d(\C)\big)$ interchanges this splitting. 

With a standard symmetry argument (see for example B.-Canary-S. \cite[Section 5.8]{pressuresurvey}), we get:
\begin{lemma}\label{degenerate} Let $\rho:\G\to\PSL_d(\R)$ be  $\{\sroot_1\}$-Anosov and let $v$ be a purely imaginary direction at $\rho.$ Then $\PP^{\omega_1}(v)=0.$ If $\rho$ is moreover $\{\sroot_2\}$-Anosov, then $\PP^{\sroot_1}(v)=0.$
\end{lemma}

\begin{proof} Let us prove on the second statement, the first one being analogous. Consider a differentiable curve $(\rho_t)_{t\in(-\eps,\eps)}\subset\Anosov_{\{\sroot_1,\sroot_2\}}(\G,\PSL_d(\C))$ such that $\rho_0=\rho,$ $\dot\rho_0=v$ and $\tau\rho_t=\rho_{-t}.$ For every conjugacy class $[\g]\in[\G],$ the functions $$t\mapsto\ell_\g(f^{\sroot_1}_{\rho_t})=(\lambda_1-\lambda_2)\big(\rho_t(\g)\big)\quad \text{ and }\quad t\mapsto h(f^{\sroot_1}_{\rho_t})$$   are invariant under $t\mapsto -t$ and are thus critical at $0.$ Consequently, for every conjugacy class, the function $t\mapsto h(f^{\sroot_1}_{\rho_t})\ell_\g(f^{\sroot_1}_{\rho_t})$ is critical at $0$ and hence Proposition \ref{0} implies that $\PP^{\sroot_1}(v)=0.$\end{proof}

\section{Pluriharmonicity of length functions and its consequences}\label{s.proofs}
In this section we prove the main results stated in the Introduction. 
\subsection{Pluriharmonic length functions} If $\rho,\eta\in\Anosov_{\Theta}(\G,\sf G_\C)$ and $\varphi\in (\sf E_\Theta)^*$ is strictly positive on $\big(\Lambda_\rho\cup\Lambda_\eta\big)-\{0\},$ then one can define their $\varphi$-\emph{dynamical intersection} by \begin{equation}\label{phi-inter}\II^\varphi(\rho,\eta)=\II(f^\varphi_\rho,f^\varphi_\eta)=\lim_{T\to\infty} \frac{1}{\#R_T(f^\varphi_\rho)}\sum_{[\g]\in R_T(f^\varphi_\rho)} \frac{\varphi\Big(\lambda\big(\eta(\g)\big)\Big)}{\varphi\Big(\lambda\big(\rho(\g)\big)\Big)} ,\end{equation} where $f^\varphi_\rho=\varphi(f^\Theta_\rho)$ is given by Corollary \ref{potentialC}.

Recall that a function is \emph{pluriharmonic} if it is locally the real part of a holomorphic function. The argument from B.-Taylor \cite[Section 5]{WP-QF} applies directly and one has the following result. 

\begin{prop}\label{pluriharmonic}Consider $\rho\in\Anosov_{\Theta}(\G,\sf G_\C)$ and $\varphi\in (\sf E_\Theta)^*$ that is strictly positive in $\Lambda_\rho-\{0\}.$ Then the function $$\II^{\varphi}_\rho=\II^\varphi(\rho,\cdot):\Anosov_{\Theta}(\G,\sf G_\C)\to\R$$ is pluriharmonic (when defined). \end{prop}

Recall from Potrie-S. \cite[Corollary 4.9]{exponentecritico} that the map $\eta\mapsto\P\big(p_\Theta(\Lambda_\eta)\big)$ is continuous on $\frak X_\Theta(\G,\sf G_\K),$ when considering the Hausdorff topology on compact subsets of $\P\big((\sf E_\Theta)^*\big).$ Thus the domain of definition of $\II^\varphi_\rho$ is an open subset of $\Anosov_\Theta(\G,\sf G_\C)$ that contains, in particular, $\rho.$ The proposition implies then that $\II^\varphi_\rho$ is (defined and) pluriharmonic on a neighborhood of $\rho.$

\begin{proof} Consider $\eta\in\Anosov_\Theta(\G,\sf G_\C)$ such that $\varphi|\Lambda_\eta-\{0\}$ is striclty positive. It follows then from Bochi-Potrie-S. \cite[Proposition 5.11]{BPS} that there exists a neighborhood $\cal U$ of $\eta$ such that the constants in Definition \ref{defAnosov} hold for every $\psi\in \cal U.$ This implies that the convergence in the definition of $\II^{\varphi}(\rho,\cdot)$ is uniform on compact subsets of its domain of definition. For each $T>0,$ the truncated sum in equation (\ref{phi-inter})  is the real part of a holomorphic function and thus Theorem 1.23 from Axler-Bourdon-Ramey \cite{harmonic} yields the result.\end{proof}

\subsection{Proof of Theorem \ref{tA}} 
Let $\rho\in\frak X\big(\pi_1S,\PSL_d(\R)\big)$ be  $(1,1,2)$-hyperconvex and assume that it is a regular point of the character variety $\frak X\big(\pi_1S,\PSL_d(\R)\big)$. 
Consider a tangent vector $v\in{\sf{T}}_\rho\frak X\big(\pi_1S,\PSL_d(\R)\big).$ Note that then $Jv$ is a purely imaginary tangent direction in ${\sf{T}}_\rho\frak X\big(\pi_1S,\PSL_d(\C)\big)$. Thus, Lemma \ref{degenerate} implies that for any $\class^1$ curve $(\rho_t)_{t\in(-\eps,\eps)}$ with $\rho_0=\rho,$ $\dot\rho_0=Jv$ and $\tau\rho_t=\rho_{-t}$ we have 
\begin{equation}\label{P=0}
0=\PP^{\sroot_1}(Jv)=\left.\frac{\partial^2}{\partial t^2}\right|_{t=0}\JJ^{\sroot_1}(\rho_0,\rho_t).
\end{equation} 

 Recall that $\II(f,f)=1$ and that if $\rho$ is $(1,1,2)$-hyperconvex then Theorem \ref{h=1} states that $h^{\sroot_1}_\rho=1$. Moreover, as observed in the proof of Lemma \ref{degenerate}, $\dot h^{\sroot_1}(\dot\rho_0)=0,$ so developing the last term of equation (\ref{P=0}) one obtains 
$$0=\Hess_\rho(h^{\sroot_1})(Jv)+\Hess_\rho\II^{\sroot_1}_\rho(Jv).$$ 
Proposition \ref{pluriharmonic} states that $\II^{\sroot_1}_\rho$ is pluriharmonic, so $\Hess_\rho\II^{\sroot_1}_\rho(Jv)=-\Hess_\rho\II^{\sroot_1}_\rho(v)$ and thus $$\Hess_\rho h^{\sroot_1}(Jv)=\Hess_\rho\II^{\sroot_1}_\rho(v).$$ 

\noindent Lemma \ref{FieldExt} implies that, at least for small $t,$ $\rho_t$ is $(1,1,2)$-hyperconvex (over $\C$) and thus Theorem \ref{t.LC} yields $h^{\sroot_1}(\rho_t)=\Hff_{\sroot_1}(\rho_t).$ Finally, since $h^{\sroot_1}\equiv1$ in a neighborhood of $\rho$ in $\frak X\big(\pi_1S,\PSL_d(\R)\big)$ one has $$\Hess_\rho\II^{\sroot_1}_\rho(v)=\PP^{\sroot_1}(v).$$ The result  follows.

\subsection{Proof of Theorem \ref{tB}}\label{proofTB} By Theorem \ref{t.LC} $\Hff_1=h^{\sroot_1}$ in a neighborhood of $\rho,$ and thus by assumption, the latter is critical at $\rho.$ Since $\JJ^{\sroot_1}(\rho,\cdot)$ is also critical at $\rho$ (Proposition \ref{JCritical}) one concludes that $\II^{\sroot_1}_\rho$ is critical at $\rho$ and thus its Hessian is well defined.

By Proposition \ref{pluriharmonic} $\II^{\sroot_1}_\rho$ is pluriharmonic and thus one has (as before) that for every $v\in\sf T_\rho\frak X\big(\G,\PSL_d(\C)\big)$ $$\Hess_\rho\II^{\sroot_1}_\rho(Jv)=-\Hess_\rho\II^{\sroot_1}_\rho(v).$$ 
One concludes that the $(+,0,-)$ signature of $\Hess_\rho\II^{\sroot_1}_\rho$ is of the form $(p,2k,p)$ for some $p\leq$ half $\dim_\R\frak X\big(\G,\PSL_d(\C)\big).$ Moreover, by Theorem \ref{HessJ} one has 
$$0\leq\PP^{\sroot_1}(Jv)=\Hess_\rho h^{\sroot_1}(Jv)-h^{\sroot_1}_\rho\Hess_\rho\II^{\sroot_1}_\rho(v),$$
 so that  $\Hess_\rho\II^{\sroot_1}_\rho(v)\geq0$ implies  $\Hess_\rho h^{\sroot_1}(Jv)\geq0.$ In particular $\Hess_\rho h^{\sroot_1}$ is positive semidefinite on a subspace of dimension at least $$\dim_\R\frak X\big(\G,\PSL_d(\C)\big)-p\geq \frac12 \dim_\R\frak X\big(\G,\PSL_d(\C)\big)$$ and the theorem is proven.

\subsection{Proof of Theorem \ref{thmComplex}}\label{prooftC}

Let $\G$ be a co-compact lattice in $\PSO(n,1)$ such that the inclusion $\iota:\G\to\PSO(n,1)$ defines, after extending coefficients, a regular point of the character variety $\frak X \big(\G,\PSU(n,1)\big)$. Theorem 2.2 (and the Remark following it) in Cooper-Long-Thistlethwaite \cite{CLT} assert that $\iota$ is then a regular point of the $\PSL_{n+1}(\R)$ character variety $\frak X\big(\G,\PSL_{n+1}(\R)\big).$ 

Moreover, since $\so(n,1)$ is the fixed point set of an involution in $\sl_{n+1}(\R),$ one has the decomposition $\sl_{n+1}(\R)=\so(n,1)\oplus\s$ with $[\s,\s]\subset\so(n,1).$ One readily sees that 
\begin{equation}\label{u}\su(n,1)=\so(n,1)\oplus i\s\subset\sl_{n+1}(\C).\end{equation} 
The twisted cohomology $H^1_\iota\big(\G,\sl_{n+1}(\R)\big)$ splits as 
$$H^1_\iota\big(\G,\sl_{n+1}(\R)\big)=H^1_\iota\big(\G,\so(n,1)\big)\oplus H^1_\iota(\G,\s).$$ 
Consequently, by equation (\ref{u}) the subspace $H^1_\iota(\G,\s)\subset H^1_\iota(\G,\sl_{n+1}(\C))$ is sent bijectively to $H^1_\iota(\G,i\s)$ when multiplied by the complex structure $J,$ i.e. 
\begin{equation}\label{j=i}J\cdot H^1_\iota(\G,\s)=H^1_\iota(\G,i\s).\end{equation}

We will need the following generalization of Crampon \cite{crampon}.

\begin{thm}[{Potrie-S. \cite[Theorem 7.2]{exponentecritico}}]\label{h<} Assume $\rho\in\frak X\big(\G,\PSL_{n+1}(\R)\big)$ has finite kernel and divides a proper open convex set of $\P(\R^{n+1}).$ Then the entropy $$h^{\omega_1}(\rho)\leq n-1$$ and equality holds only if $\rho$ has values in $\PSO(n,1).$
\end{thm}
This has the following useful consequence.
\begin{cor}\label{P>0}The spectral radius pressure form $\PP^{\omega_1}$ on $\frak X\big(\G,\PGL_{n+1}(\R)\big) $ is non-degenerate at $\iota.$ 
\end{cor}

\begin{proof}
When $n=2$ this follows directly from Theorem \ref{t.pressure1}, but if  $n>2,$ the embedding $\so(n,1)\subset\sl_{n+1}(\R)$ is not $\PGL_{n+1}(\R)$-generic so, even though $\iota(\G)$ is irreducible, we need additional arguments. Nevertheless, by Theorem \ref{h<}, the entropy function $\rho\mapsto h^{\omega_1}(\rho)$ is critical at $\iota,$ so by Proposition \ref{0} one only needs to verify that the set 
$$\big\{d_\iota\omega_1^\g:[\g]\in[\G]\big\}$$
spans the cotangent space $\sf{T}^*_\iota\frak X\big(\G,\PSL_{n+1}(\R)\big),$
 where $\omega_1^\g:\frak X\big(\G,\PGL_{n+1}(\R)\big)\to\R$ is the function 
 $$\rho\mapsto\omega_1\Big(\lambda\big(\rho(\g)\big)\Big).$$  
 As $\iota$ is irreducible and projective Anosov, this is the content of B.-Canary-Labourie-S. \cite[Proposition 10.1]{pressure}.
\end{proof}

Consider then $v\in H^1_\iota(\G,\s)\subset \sf T_\iota\frak X\big(\G,\PSL_{n+1}(\R)\big),$ by equation (\ref{j=i}) the purely imaginary vector $J\cdot v\in\sf T_\iota\frak X\big(\G,\PSL_{n+1}(\C)\big)$ belongs to $ H^1_\iota\big(\G,\PSU(n,1)\big)$ and represents thus a non-trivial infinitesimal deformation of $\iota$ inside $\PSU(n,1).$ As in Lemma \ref{degenerate} we choose a differentiable curve $(\rho_t)_{t\in(-\epsilon,\epsilon)}\subset\frak X\big(\G,\PSU(n,1)\big) $  with $\rho_0=\iota$ and $\dot \rho_0=Jv$ and $\tau\rho_t=\rho_{-t}$.

 By Lemma \ref{degenerate} we have that 
\begin{equation}\label{P1=0}
0=\PP^{\omega_1}_\iota(Jv)=\left.\frac{\partial^2}{\partial t^2}\right|_{t=0}\JJ^{\omega_1}(\iota,\rho_t).
\end{equation} 
Expanding the second term, and using that both $h^{\omega_1}(\rho_t)$ and $\II^{\omega_1}_\iota(\rho_t)$ are critical at $t=0$ (as in the proof of Lemma \ref{degenerate}) and that $\II^{\omega_1}_\iota$ is pluriharmonic, we get
$$0=\Hess_\iota(h^{\omega_1})(Jv)-(n-1)\Hess_\iota(\II^{\omega_1}_\iota)(v).$$ 
On the other hand 
$$\PP^{\omega_1}_\iota(v)=\Hess_\iota(h^{\omega_1})(v)+(n-1)\Hess_\iota(\II^{\omega_1}_\iota)(v).$$
Which in turn gives
$$\Hess_\iota(h^{\omega_1})(Jv)=\PP^{\omega_1}_\iota(v)-\Hess_\iota(h^{\omega_1})(v)>0,$$
since $\PP^{\omega_1}_\iota(v)>0$ by Corollary \ref{P>0}, and $-\Hess_\iota(h^{\omega_1})(v)\geq0$ since by Theorem \ref{h<} $\iota$ is a global maxima of $h^{\omega_1}$ among deformations in $\PSL_{n+1}(\R).$ The result then follows.

\subsection{The Hessian of the entropy at the Fuchsian locus of the Hitchin component}\label{s.5.5}

Applying the same techniques as in the last section we can also show the following result on the Hitchin component.

\begin{cor}\label{cC}
Let $\iota\in\hitchin_d(S)$ be a representation $\pi_1S\to\PSL_2(\R)\to \PSL_d(\R)$ in the embedded Teichhm\"uller space. Then $\Hess(h_\iota^{\omega_1})$ is positive definite on purely imaginary directions of $\sf T_\iota\frak X\big(\pi_1S,\PSL_d(\C)\big).$ \end{cor}

\begin{proof}
We mimic the last paragraph. In this case the pressure form $\PP^{\omega_1}$ is positive definite on $\sf T_\iota \hitchin\big(S,\PSL_d(\R)\big)$ directly by Theorem \ref{t.pressure1}. One gets, through the same arguments, that 
$$\Hess_\rho(h^{\omega_1})(Jv)=\PP^{\omega_1}(v)-\Hess(h^{\omega_1})(v).$$
As we already observed, the first term on the right hand side is positive by Theorem \ref{t.pressure1}, while $\Hess(h^{\omega_1})(v)\leq0$ since, by Potrie-S. \cite[Theorem A]{exponentecritico}, Fuchsian representations are maxima for the entropy within the Hitchin locus. The corollary follows.
\end{proof}

We refer the reader to Dey-Kapovich \cite{KD} (see also Ledrappier \cite{ledrappier} and Link \cite{link}) for an interpretation of the critical exponent $h^{\omega_1}(\rho)$ as the Hausdorff dimension of the limit set with respect to a visual metric, i.e. a metric with respect to which the group action is conformal.

Finally, it would be interesting to relate Corollary \ref{cC}, or an analog of it, to the recent work by Dai-Li \cite{Dai-Li} studying the translation lengths on the symmetric space of $\PSL_d(\C),$ when one deforms a Fuchsian representation along its Hitchin fiber.

\bibliography{new}
\bibliographystyle{plain}

\author{\vbox{\footnotesize\noindent 
	Martin Bridgeman\\
	Boston College \\ Department of Mathematics\\ 
	Chestnut Hill, Ma 02467 Unites States of America\\
	\texttt{bridgem@bc.edu}
\bigskip}}

\author{\vbox{\footnotesize\noindent 
	Beatrice Pozzetti\\
Ruprecht-Karls Universit\"at Heidelberg\\ Mathematisches Institut, Im
Neuenheimer Feld 205, 69120 Heidelberg, Germany\\
	\texttt{pozzetti@mathi.uni-heidelberg.de}
\bigskip}}

\author{\vbox{\footnotesize\noindent 
	Andr\'es Sambarino\\
	Sorbonne Universit\'e \\ IMJ-PRG (CNRS UMR 7586)\\ 
	4 place Jussieu 75005 Paris France\\
	\texttt{andres.sambarino@imj-prg.fr}
\bigskip}}

\author{\vbox{\footnotesize\noindent 
	Anna Wienhard\\
	Ruprecht-Karls Universit\"at Heidelberg\\ Mathematisches Institut, Im
Neuenheimer Feld 205, 69120 Heidelberg, Germany\\ HITS gGmbH, Heidelberg Institute for Theoretical Studies\\ Schloss-Wolfsbrunnenweg 35, 69118 Heidelberg, Germany\\
	\texttt{wienhard@uni-heidelberg.de}
\bigskip}}

\end{document}